\documentclass[reqno,10pt]{amsart}

\usepackage{amsmath}
\usepackage{amssymb}
\allowdisplaybreaks[4]

\newtheorem{thm}{Theorem}[section]
\newtheorem{prop}[thm]{Proposition}
\newtheorem{lem}[thm]{Lemma}
\newtheorem{q}[thm]{Question}
\newtheorem{cor}[thm]{Corollary}
\newtheorem{conj}[thm]{Conjecture}
\newtheorem*{claim}{Claim}
\newtheorem*{fact}{Fact}

\theoremstyle{definition}
\newtheorem{definition}[thm]{Definition}
\newtheorem{example}[thm]{Example}

\theoremstyle{remark}
\newtheorem{remark}[thm]{Remark}

\numberwithin{equation}{section}

\begin{document}

\title{Bounding the volumes of singular weak log del Pezzo surfaces}
\author{Chen Jiang}
\address{Graduate School of Mathematical Sciences, the University of Tokyo,
3-8-1 Komaba, Meguro-ku, Tokyo 153-8914, Japan.}
\email{cjiang@ms.u-tokyo.ac.jp}

\begin{abstract}
We give an optimal upper bound for the anti-canonical volume of an
$\epsilon$-lc weak log del Pezzo surface. Moreover, we consider the
relation between the bound of the volume and the Picard number of
the minimal resolution of the surface. Furthermore we consider blowing up
several points on a Hirzebruch surface in general position and give
some examples of smooth weak log del Pezzo surfaces.
\end{abstract}

\maketitle

\tableofcontents

\section{Introduction}
Throughout this article, we work over an algebraically closed field
of arbitrary characteristic. We will freely use the standard notations in \cite{KM}. We start by some basic definitions.

\begin{definition}
Let $X$ be a normal projective surface and $\Delta$ be an
$\mathbb{R}$-divisor on $X$ with coefficients in $[0,1]$ such that
$K_X+\Delta$ is $\mathbb{R}$-Cartier. We say that $(X, \Delta)$ is a
\emph{weak log del Pezzo surface} if $-(K_X+\Delta)$ is nef and big.
\end{definition}

\begin{definition}
Let $X$ be a normal projective variety and let $\Delta$ be an
$\mathbb{R}$-divisor on $X$  such that $K_X+\Delta$ is
 $\mathbb{R}$-Cartier. Let $f: Y\rightarrow X$ be a log
resolution of $(X, \Delta)$, write
$$
K_Y =f^*(K_X+\Delta)+\sum a_iF_i,
$$
where $F_i$ is a prime divisor. For some $\epsilon \in (0,1]$, the
pair $(X,\Delta)$ is called

(a) \emph{$\epsilon$-kawamata log terminal} (\emph{$\epsilon$-klt},
for short) if $a_i> -1+\epsilon$ for all $i$, or

(b) \emph{$\epsilon$-log canonical} (\emph{$\epsilon$-lc}, for
short) if $a_i\geqslant -1+\epsilon$ for all $i$.
\end{definition}

In this article we first prove the following theorem.

\begin{thm}\label{main thm}
Let $(X,\Delta)$ be an $\epsilon$-lc weak log del Pezzo surface.
Then the anti-canonical volume
${\rm Vol}(-(K_X+\Delta))=(K_X+\Delta)^2$ satisfies
$$
(K_X+\Delta)^2\leqslant \max\bigg\{9, \lfloor 2/\epsilon \rfloor+
4+\frac{4}{\lfloor 2/\epsilon \rfloor}\bigg\},
$$
where $\lfloor~\rfloor$ means round down.

Moreover, the equality holds if and only if one of the
following holds:

(1) $\epsilon >\frac{2}{5}$ and  $(X,\Delta)$ is $(\mathbb{P}^2,
0)$.

(2) $\epsilon \leqslant \frac{1}{2}$ and $(X,\Delta)$ is
$(\mathbb{F}_n, (1-\frac{2}{n})S_n)$ or $(\mbox{PC}_n, 0)$, where
$n=\lfloor 2/\epsilon \rfloor$, $\mathbb{F}_n$ is the $n$-th
Hirzebruch surface, $S_n\subset \mathbb{F}_n$ is the unique curve
with negative self-intersection and $\mbox{PC}_n$ is the projective
cone over a rational normal curve of degree $n$;
\end{thm}

\begin{remark}
By the examples, we can see that this bound is also an optimal bound for $\epsilon$-lc log del Pezzo surface or  $\epsilon$-lc log del Pezzo surface with Picard number one.
\end{remark}

The motivation of this kind of problem is the following B-A-B
Conjecture due to A. Borisov, L. Borisov and V. Alexeev.
\begin{definition}
Let $X$ be a normal projective variety and $\Delta$ be a
$\mathbb{Q}$-divisor on $X$ with coefficients in $[0,1]$ such that
$K_X+\Delta$ is $\mathbb{Q}$-Cartier. We say that $(X, \Delta)$ is a
\emph{log $\mathbb{Q}$-Fano variety} if $-(K_X+\Delta)$ is ample.
\end{definition}

\begin{definition}
A collection of varieties $\{X_\lambda\}_{\lambda\in \Lambda}$ is
said to be \emph{bounded} if there exists $h:\mathcal{X}\rightarrow
S$ a morphism of finite type of Neotherian schemes such that for
each $X_\lambda$, $X_\lambda\simeq \mathcal{X}_s$ for some $s\in S$.
\end{definition}

\begin{conj}[Borisov-Alexeev-Borisov]
Fix $0<\epsilon<1$, an integer $n>0$, and consider the set of all
$n$-dimensional $\epsilon$-klt log $\mathbb{Q}$-Fano varieties $(X,
\Delta)$. The set of underlying varieties $\{X\}$ is bounded.
\end{conj}

The B-A-B Conjecture is still open in dimension three and higher. We
are mainly interested in the following weak conjecture for
anti-canonical volumes which is a consequence of B-A-B Conjecture.

\begin{conj}[Boundedness of anti-canonical volumes]\label{conj2}
Fix $0<\epsilon<1$, an integer $n>0$, and consider the set of all
$n$-dimensional $\epsilon$-klt log $\mathbb{Q}$-Fano varieties $(X,
\Delta)$. The volume ${\rm Vol}(-(K_X+\Delta))=(-(K_X+\Delta))^n$ is
bounded from above by a fixed number $M(n,\epsilon)$ depending only
on $n$ and $\epsilon$.
\end{conj}

In dimension two, Conjecture \ref{conj2} is well-researched in
history. Alexeev establishes two dimensional B-A-B Conjecture and the boundedness for the anti-canonical volumes  in \cite{A}, but no clear bound is written down. Alexeev and Mori in \cite{AM} give a simplified argument for two dimensional B-A-B Conjecture and give an upper bound for the pair
$(X, \Delta)$ which is
$$
(K_X+\Delta)^2\leqslant (\lfloor 2/\epsilon \rfloor+ 2)^2.
$$
Recently Lai in \cite{Lai} gives an upper bound using 
covering families of tigers of M$^c$Kernan and Connectedness Lemma of Koll\'{a}r and Shokurov, which turns out to be
$$
(K_X+\Delta)^2\leqslant \max\bigg\{64, \frac{8}{\epsilon}+4\bigg\}.
$$
Here we remark that one could use Lai's method to get a refinement
of his bound by carefully computation, which is
$$
(K_X+\Delta)^2\leqslant \max\bigg\{36, \lfloor 2/\epsilon \rfloor+
4+\frac{4}{\lfloor 2/\epsilon \rfloor}\bigg\}.
$$
This bound is very close to be optimal. But we should also remark that
Lai's method only works for complex number field and rational
boundary and the method used in this article is totally different
from Lai's method.

For dimension three, recently in \cite{Lai}, Lai proves the
following theorem:
\begin{thm}[\cite{Lai}, Theorem B]
Let $(X, \Delta)$ be an $\epsilon$-klt $\mathbb{Q}$-factorial log
$\mathbb{Q}$-Fano threefold of $\rho(X)=1$. The degree $-K_X^3$
satisfies
$$
-K_X^3\leqslant
\Big(\frac{24M(2,\epsilon)R(2,\epsilon)}{\epsilon}+12\Big)^3,
$$
where $R(2, \epsilon)$ is an upper bound of the Cartier index of
$K_S$ for S any $\epsilon/2$-klt log del Pezzo surface of $\rho(S) =
1$ and $M(2, \epsilon)$ is an upper bound of the volume Vol$(-K_S) =
K_S^2$ for S any $\epsilon/2$-klt log del Pezzo surface of $\rho(S)
= 1$.
\end{thm}
This is another motivation to find an optimal bound in dimension
two. Recently the author is informed by Lai that the assumption on
$\rho(X)$ in the above theorem is removed. Hence Conjecture \ref{conj2} is solved in dimension three. But it remains open in dimension four and higher.

In this article, we solve Conjecture \ref{conj2} in dimension two by
giving an optimal bound.

In the proof of Theorem \ref{main thm}, we note that when
$\rho({X_\textrm{min}})$ increases, the upper bound of
$\textrm{Vol}(-(K_X+\Delta))$ decreases (cf. Remark \ref{rem1}),
where ${X_\textrm{min}}$ is the minimal resolution of $(X, \Delta)$.
So we consider the relation between $\rho({X_\textrm{min}})$ and the
upper bound of $\textrm{Vol}(-(K_X+\Delta))$. We give optimal bounds
for the cases when $\rho({X_\textrm{min}})$ gets larger.

\begin{thm}\label{r3}
Let $(X,\Delta)$ be an $\epsilon$-lc weak log del Pezzo surface with
$\rho({X_\textrm{min}})\geqslant 3$. Then the anti-canonical volume
$\textrm{Vol}(-(K_X+\Delta))=(K_X+\Delta)^2$ satisfies
$$
(K_X+\Delta)^2\leqslant \lfloor 2/\epsilon \rfloor+
3+\frac{4}{\lfloor 2/\epsilon \rfloor},
$$
which is optimal.
\end{thm}

\begin{thm}\label{r4}
Let $(X,\Delta)$ be an $\epsilon$-lc weak log del Pezzo surface with
$\rho({X_\textrm{min}})\geqslant 4$. Then the anti-canonical volume
$\textrm{Vol}(-(K_X+\Delta))=(K_X+\Delta)^2$ satisfies
$$
(K_X+\Delta)^2\leqslant \max\bigg\{\lfloor 2/\epsilon \rfloor+
2+\frac{4}{\lfloor 2/\epsilon \rfloor}, \lfloor
(3+\epsilon)/2\epsilon \rfloor+\frac{5}{2}+\frac{9}{4\lfloor
(3+\epsilon)/2\epsilon \rfloor-2} \bigg\},
$$
which is optimal.
\end{thm}

Since almost all smooth  weak log del Pezzo surfaces
come from blowing up the Hirzebruch surfaces (cf. Lemma \ref{rat
lem}), we consider the case when blowing up the Hirzebruch surfaces
at points in general position.

\begin{thm}[=Theorem \ref{general blowup}]
Let $X$ be a smooth surface which is given by blowing up $k$ points
on $\mathbb{F}_n$. Assume that these points are not on $S_n$, and no
two of them are on the same fiber of $\mathbb{F}_n \rightarrow \mathbb{P}^1$. Assume that $-(K_X+\Delta)$ is
nef and big, where $\Delta$ is an effective $\mathbb{R}$-divisor (not neccesary a boundary). Then the anti-canonical volume
${\rm Vol}(-(K_X+\Delta))=(K_X+\Delta)^2$ satisfies
$$
(K_X+\Delta)^2\leqslant\begin{cases}
n+4+\frac{4}{n}-k & \mbox{if } n\geqslant 2;\\
8-k  & \mbox{if } n=0,1,
\end{cases}
$$
which is optimal if the points are in general position.
\end{thm}

Finally we give some examples of smooth  weak log del Pezzo surfaces
which make the above bounds optimal.

\begin{thm}[=Theorem \ref{example2}]
Let $X$ be a smooth surface which is given by blowing up $k$ points
on $\mathbb{F}_n$, $n\geqslant 2$. Assume that these points are not
on $S_n$, and no two of them are on the same fiber of $\mathbb{F}_n \rightarrow \mathbb{P}^1$. $S_{n}^\prime$
is the strict transform of $S_{n}$. Then

(1)  if $k\leqslant n+2$, then $(X,(1-\frac{2}{n})S_n^\prime)$ is a
weak log del Pezzo surface;

(2) if $n+2< k< \frac{1}{n}(n+2)^2$ and the points are in general
position, then $(X,(1-\frac{2}{n})S_n^\prime)$ is a weak log del
Pezzo surface;

(3) if $k\geqslant \frac{1}{n}(n+2)^2$, then $(X,\Delta)$ can not be
a weak log del Pezzo surface for any boundary $\Delta$.
\end{thm}

\section{Preliminaries}

\subsection{Minimal resolution}
Let $(X, \Delta)$ be an $\epsilon$-lc weak log del Pezzo surface.
The minimal resolution $\pi: {X_\textrm{min}}\rightarrow X$ of $(X,
\Delta)$ is the unique proper birational morphism such that
${X_\textrm{min}}$ is a smooth projective surface and
$K_{{X_\textrm{min}}}+\Delta_{{X_\textrm{min}}}=\pi^*(K_X+\Delta)$
for some effective $\mathbb{R}$-divisor $\Delta_{{X_\textrm{min}}}$
on ${X_\textrm{min}}$. Note that minimal resolutions always exist
for two-dimensional log pairs.

\subsection{Hirzebruch surfaces and projective cones}
Hirzebruch surfaces play important roles in this article. We
recall some basic properties of the Hirzebruch surfaces
$\mathbb{F}_n=\mathbb{P}_{\mathbb{P}^1}(\mathcal{O}_{\mathbb{P}^1}\oplus\mathcal{O}_{
\mathbb{P}^1}(n))$, $n\geqslant 0$. Denote by $h$ (resp. $f$) the
class in $\textrm{Pic }\mathbb{F}_n$ of the tautological bundle
$\mathcal{O}_{\mathbb{F}_n}(1)$ (resp. of a fiber). Then
$\textrm{Pic }\mathbb{F}_n=\mathbb{Z}h\oplus \mathbb{Z}f$ with
$f^2=0$, $f.h=1$, $h^2=n$. If $n>0$, there is a unique
irreducible curve $S_n\subset \mathbb{F}_n$ such that $S_n \equiv h-nf$, $S_n^2=-n$. For
$n=0$, we can also choose one curve whose class in $\textrm{Pic
}\mathbb{F}_0$ is $h$ and denote it by $S_0$. We often denote by $F$ a fiber  of $\mathbb{F}_n \rightarrow \mathbb{P}^1$ and by $F_p$ the fiber passing through some point $p\in \mathbb{F}_n$.

The following easy proposition allows us to choose the position of
the blow-up center on Hirzebruch surfaces.
\begin{prop}\label{blowupposition}
$\mathbb{F}_n$ blown up at a point on $S_n$ is isomorphic to
$\mathbb{F}_{n+1}$ blown up at a point not on $S_{n+1}$.
\end{prop}
\begin{proof}
Blowing up a point on $S_n\subset \mathbb{F}_n$, and blowing down
the strict transform of the fiber, we get exactly
$\mathbb{F}_{n+1}$.
\end{proof}

Now we consider $\mbox{PC}_n$, the projective cone over a rational
normal curve of degree $n \geqslant2$ with the unique singular point
$O\in \mbox{PC}_n$. Blowing up $O\in \mbox{PC}_n$, we have a birational morphism
$\phi: \mathbb{F}_n \rightarrow \textrm{PC}_n$, and
$$
K_{\mathbb{F}_n}+\Big(1-\frac{2}{n}\Big)S_n=\phi^*(K_{\textrm{PC}_n}).
$$
Hence the minimal resolution of
$(\mbox{PC}_n, 0)$ is $(\mathbb{F}_n, (1-\frac{2}{n})S_n)$, which
makes the examples in Theorem \ref{main thm} very natural. It is
easy to see that $\mbox{PC}_n$ is $\mathbb{Q}$-factorial of Picard
number one with $-K_{\textrm{PC}_n}$ ample.

\subsection{Intersection inequalities}
In this section we recall some easy inequalities on intersection
number of curves on surfaces, which is frequently used in this
article.

Assume that $C$ and $D$ are divisors on a smooth surface $X$ having no
common irreducible component, then by \cite[Proposition V.1.4]{H},
$$
C.D=\sum_{P\in C\cap D}(C.D)_P,
$$
where $(C.D)_P$ is the intersection multiplicity of $C$ and $D$ at $P$. And by easy calculation (cf. \cite[Exercise I.5.4(a)]{H}), for a point $p$, we have
$$
(C.D)_p\geqslant \mbox{mult}_p(C)\cdot\mbox{mult}_p(D).
$$

In this article, we will only use a special case when $D$ is an
irreducible curve with $\mbox{mult}_{p_i}(D)=1$ for $p_i \in D$, $1\leqslant i
\leqslant k$. In this case, we have
\begin{align*}
 C.D
\geqslant{}& \sum_{P\in C\cap D}(C.D)_P\geqslant \sum_{p_i\in C}(C.D)_{p_i}\\
\geqslant{}& \sum_{p_i\in C}
\mbox{mult}_{p_i}(C)\cdot\mbox{mult}_{p_i}(D)\\
={}& \sum_{p_i\in C} \mbox{mult}_{p_i}(C)\\
={}& \sum_{i=1}^k \mbox{mult}_{p_i}(C).
\end{align*}
We will use this inequality frequently without mention.

\section{Bounding the volumes I}

In this section, we prove Theorem \ref{main thm}.

Let $(X,\Delta)$ be an $\epsilon$-lc weak log del Pezzo surface.
Take the minimal resolution $\pi: ({X_\textrm{min}},
\Delta_{X_\textrm{min}}) \rightarrow (X, \Delta)$ with
$K_{{X_\textrm{min}}}+\Delta_{{X_\textrm{min}}}=\pi^*(K_X+\Delta)$.
It is easy to see that $({X_\textrm{min}}, \Delta_{X_\textrm{min}})$
is a weak log del Pezzo surface, $\Delta_{X_{\textrm{min}}}$ is with
coefficients in $[0, 1-\epsilon]$ and we have
$$
\textrm{Vol}(-(K_{{X_\textrm{min}}}+\Delta_{{X_\textrm{min}}}))=(K_{{X_\textrm{min}}}+\Delta_{{X_\textrm{min}}})^2=(K_X+\Delta)^2=\textrm{Vol}(-(K_X+\Delta
)).
$$
Hence replacing $(X, \Delta)$ by $({X_\textrm{min}},
\Delta_{X_\textrm{min}})$, we may assume that $X$ is smooth and
$\Delta$ is with coefficients in $[0, 1-\epsilon]$ and $-(K_X+\Delta)$ is nef and big.

The following lemma gives a rough classification of smooth weak log
del Pezzo surfaces.

\begin{lem}[{cf. \cite[Lemma 1.4]{AM}}]\label{rat lem}
If $X$ is a smooth surface and $\Delta$ is with coefficients in $[0,
1-\epsilon]$ and $-(K_X+\Delta)$ is nef and big, then $X$ is rational,
and either $X\simeq \mathbb{P}^2$ or there exists a birational
morphism $g:X\rightarrow \mathbb{F}_n$ with $n\leqslant 2/\epsilon$.
\end{lem}

\begin{proof}
By Base Point Free Theorem (cf. \cite[Theorem 7.1]{HM}),
$-(K_X+\Delta)$ is semi-ample. Hence there exists an effective
$\mathbb{R}$-divisor $D$ which is nef and big such that
$K_X+\Delta+D\equiv0$ and $\Delta+D$ is with coefficients in $[0,
1-\epsilon]$.

Now consider the pair $(X,B=\Delta+D)$ and write $B=\sum b_jB_j$.
$K_X\equiv -B\neq 0$. Assuming $X\neq \mathbb{P}^2$, we can contract
$(-1)$-curves on $X$ and its contractions until we get a birational
morphism $g: X\rightarrow \overline{X}$ to a model $\overline{X}$
which is a $\mathbb{P}^1$-bundle over a smooth curve $C$.
Denote by $\overline{B}_j$ (resp. $\overline{B}$) the image of $B_j$
(resp. B) on $\overline{X}$. Now
$K_{\overline{X}}+\overline{B}=g_*(K_X+B)\equiv 0$.

If $g(C)>0$ and there exists a curve $E$ on $\overline{X}$ with
$E^2<0$, then
\begin{align*}
-2\leqslant{}&
2p_a(E)-2=(K_{\overline{X}}+E).E\\
={}&\epsilon E^2+(K_{\overline{X}}+(1-\epsilon)E).E\\
\leqslant{}&\epsilon E^2+(K_{\overline{X}}+\overline{B}).E=\epsilon
E^2<0,
\end{align*}
which implies that $E$ is a smooth rational curve. Since it does not
lie in a fiber of $\overline{X}\rightarrow C$, $g(C)=0$, which is a
contradiction.

If $g(C)>0$ and $E^2\geqslant0$ for all  curves $E$ on
$\overline{X}$, then
$$
0\geqslant  8-8g(C)=K_{\overline{X}}^2=\overline{B}^2\geqslant0.
$$
It follows that all
$\overline{B}_j^2=\overline{B}_j.\overline{B}_k=0$, in particular,
$\overline{D}^2=0$, where $\overline{D}$ is the image of $D$. Since
blow-downs preserve nefness and bigness, $\overline{D}$ is nef and
big, which is a contradiction.

Hence $g(C)=0$ and $\overline{X}=\mathbb{F}_n$ for some $n$. Since
\begin{align*}
-2={}&
2p_a(S_n)-2=(K_{\overline{X}}+S_n).S_n\\
={}&\epsilon S_n^2+(K_{\overline{X}}+(1-\epsilon)S_n).S_n\\
\leqslant{}&\epsilon
S_n^2+(K_{\overline{X}}+\overline{B}).S_n=\epsilon S_n^2=-\epsilon
n,
\end{align*}
we have $n\leqslant 2/\epsilon$.
\end{proof}

Hence at this time $X$ must be rational. We can analyze it case by
case.

\textbf{Case 1:  $X\simeq \mathbb{P}^2.$}

In this case, $-K_X$ is ample with $-K_X\geqslant -(K_X+\Delta)$.
Hence
$$
(K_X+\Delta)^2\leqslant K_X^2=9.
$$
Moreover, the equality holds if and only if $\Delta=0$, i.e., $(X,
\Delta)\simeq(\mathbb{P}^2, 0)$.

\textbf{Case 2: $X\simeq \mathbb{F}_n$, $n\leqslant \lfloor
2/\epsilon \rfloor$.}

Write $\Delta=aS_n+\sum d_i \Delta_i$, where $\Delta_i$ is a reduced
irreducible curve  different from $S_n$ and $\Delta_i\equiv
\alpha_i h+\beta_i f$ with $\alpha_i,\beta_i \geqslant 0$.
Since $K_X+\Delta=K_X+aS_n+\sum d_i \Delta_i$ is anti-nef, we have
\begin{align}\label{3.1}
0\geqslant (K_X+\Delta).S_n=n-2-na+\sum d_i \beta_i  \geqslant
n-2-na;
\end{align}
\begin{align}\label{3.2}
0\geqslant (K_X+\Delta).f=-2+a+\sum d_i \alpha_i  \geqslant -2+a.
\end{align}
Now $K_X+\Delta=(-2+a+\sum d_i \alpha_i)h+(n-2-na+\sum d_i
\beta_i)f$, we have
\begin{align*}
{}& (K_X+\Delta)^2 \\
= {}&  \left(\left(-2+a+\sum d_i \alpha_i\right) h+ \left(n-2-na+\sum d_i \beta_i\right) f\right)^2\\
= {}&  n\left(-2+a+\sum d_i \alpha_i\right)^2+2\left(-2+a+\sum d_i \alpha_i\right) \left(n-2-na+\sum d_i \beta_i\right)\\
\leqslant {}&  n(-2+a)^2+2(-2+a)(n-2-na)\\
= {}&  8+(2n-4)a-na^2 \\
\leqslant{}&
\begin{cases}
n+4+\frac{4}{n} & \mbox{if } n\geqslant 2\\
8  & \mbox{if } n=0,1
\end{cases}\\
 \leqslant {}&  \lfloor 2/\epsilon \rfloor+ 4+\frac{4}{\lfloor
2/\epsilon \rfloor}.
\end{align*}

The first inequality follows from (\ref{3.1}) and (\ref{3.2}). The
equality holds if and only if $\alpha = -2+a$ and $\beta = n-2-na$,
if and only if $\sum d_i \alpha_i=\sum d_i \beta_i=0$, if and only
if $\sum d_i\Delta=0$.
The second inequality follows from the monotonic property of the
quadratic function of $a$. The equality holds if and only if
$a=\{0, 1-\frac{2}{n}\}$.
The third inequality follows from the monotonic property of the first
function on the range of $n$ which is $2\leqslant n \leqslant
\lfloor 2/\epsilon \rfloor$ and the fact that $\lfloor 2/\epsilon
\rfloor+ 4+\frac{4}{\lfloor 2/\epsilon \rfloor}\geqslant 8$. The
equality holds if and only if $n = \lfloor 2/\epsilon \rfloor$ or  $\lfloor 2/\epsilon \rfloor$=2 and $n=0,1$.

Hence in this case, we have
$$
(K_X+\Delta)^2 \leqslant \lfloor 2/\epsilon \rfloor+
4+\frac{4}{\lfloor 2/\epsilon \rfloor}.
$$
The equality holds if and only if one of the followings holds: (i) $(X,\Delta)=(\mathbb{F}_n,
(1-\frac{2}{n})S_n)$ where $n=\lfloor 2/\epsilon \rfloor$ or (ii)  $\lfloor 2/\epsilon
\rfloor$=2 and $(X,\Delta)=(\mathbb{F}_0, 0)$ or $(\mathbb{F}_1, 0)$.

\textbf{Case 3: there exists a nontrivial birational morphism
$g:X\rightarrow \mathbb{F}_n$ with $n\leqslant \lfloor 2/\epsilon
\rfloor$.}

By the proof of Lemma \ref{rat lem}, we know that $g$ is decomposed
by blow-downs of $(-1)$-curves. Since push-forward of blow-downs preserve nef curves,
$K_{\mathbb{F}_n}+g_*(\Delta)=g_*(K_X+\Delta)$ is anti-nef. Then by
Case 2, we have
$$
(K_{\mathbb{F}_n}+g_*(\Delta))^2\leqslant  \lfloor 2/\epsilon
\rfloor+ 4+\frac{4}{\lfloor 2/\epsilon \rfloor}.
$$
By Negativity Lemma (cf. \cite[Lemma 3.39]{KM}), we have
\begin{align}\label{star}
K_X+\Delta=g^*(K_{\mathbb{F}_n}+g_*(\Delta))+E,
\end{align}
where $E$ is an effective $g$-exceptional divisor. Then
\begin{align*}
{}&(K_X+\Delta)^2\\
= {}& (g^*(K_{\mathbb{F}_n}+g_*(\Delta))+E)(K_X+\Delta)\\
\leqslant {}& g^*(K_{\mathbb{F}_n}+g_*(\Delta))(K_X+\Delta)\\
= {}& (K_{\mathbb{F}_n}+g_*(\Delta))^2\\
\leqslant {}& \lfloor 2/\epsilon \rfloor+ 4+\frac{4}{\lfloor
2/\epsilon \rfloor}.
\end{align*}

The first inequality follows from the anti-nefness of $K_X+\Delta$ and
the effectiveness of $E$. The equality holds if and only if $E=0$.

Hence
$$
(K_X+\Delta)^2   \leqslant  \lfloor 2/\epsilon \rfloor+
4+\frac{4}{\lfloor 2/\epsilon \rfloor},
$$
where the equality holds only if there is a nontrivial crepant
birational morphism $g: (X,\Delta)\rightarrow (\mathbb{F}_n,
(1-\frac{2}{n})S_n)$ for some $n\geqslant 2$ or $g:
(X,\Delta)\rightarrow (\mathbb{F}_0, 0)$ or $g:
(X,\Delta)\rightarrow (\mathbb{F}_1, 0)$ obtained by blow-ups at points. But it
is easy to see that after non-trivial crepant blow-ups on the pair
$(\mathbb{F}_n, (1-\frac{2}{n})S_n)$ or $(\mathbb{F}_0, 0)$ or
$(\mathbb{F}_1, 0)$, $\Delta$ can not be effective.
Hence in this case,
$$
(K_X+\Delta)^2 < \lfloor 2/\epsilon \rfloor+ 4+\frac{4}{\lfloor
2/\epsilon \rfloor}.
$$

Summing up the above cases, we almost complete the proof of Theorem
\ref{main thm}. Recall that we take minimal
resolution at the beginning, now we come back to the singular case,
actually we have proved the inequality
$$
(K_X+\Delta)^2 \leqslant \lfloor 2/\epsilon \rfloor+
4+\frac{4}{\lfloor 2/\epsilon \rfloor}.
$$
And the equality holds if and only if one of the following
holds:

(1) $\epsilon \leqslant \frac{1}{2}$ and the minimal resolution of
$(X,\Delta)$ is $(\mathbb{F}_n, (1-\frac{2}{n})S_n)$, where
$n=\lfloor 2/\epsilon \rfloor$;

(2) $\epsilon >\frac{2}{5}$ and the minimal resolution of
$(X,\Delta)$ is $(\mathbb{P}^2, 0)$.

To finish the proof, we only need the next claim.
\begin{claim}
(1) If the minimal resolution of $(X,\Delta)$ is $(\mathbb{F}_n,
(1-\frac{2}{n})S_n)$ with $n\geqslant2$, then $(X,\Delta)$ is
$(\mathbb{F}_n, (1-\frac{2}{n})S_n)$ or  $(\textrm{PC}_n, 0)$.

(2) If the minimal resolution of $(X,\Delta)$ is $(\mathbb{P}^2,
0)$, then $(X,\Delta)$ is $(\mathbb{P}^2, 0)$.
\end{claim}

\begin{proof}
Let $\pi:(\mathbb{F}_n, (1-\frac{2}{n})S_n)\rightarrow
(X,\Delta)$ be the minimal resolution. Since $-(K_X+\Delta)$ is
semi-ample, $\pi$ is a semi-ample model of $-(K_{\mathbb{F}_n}+
(1-\frac{2}{n})S_n)$ (cf. \cite[Definition 3.6.4]{BCHM}).
Let $\phi:(\mathbb{F}_n, (1-\frac{2}{n})S_n)\rightarrow
(\textrm{PC}_n, 0)$ be the  minimal resolution. Since
$-K_{\textrm{PC}_n}$ is ample, $\phi$ is the ample model of
$-(K_{\mathbb{F}_n}+ (1-\frac{2}{n})S_n)$ (cf. \cite[Definition
3.6.4]{BCHM}).
Hence by \cite[Lemma 3.6.5]{BCHM}, there is a birational morphism
$\mu: X\rightarrow \textrm{PC}_n$, with $\phi=\mu\circ\pi$.
Since  $\phi$ only contracts one curve $S_n$, we conclude that
either $\mu$ or $\pi$ is an isomorphism.

(2) follows from similar argument.
\end{proof}

Then the proof of Theorem \ref{main thm} is completed.

\begin{remark}\label{rem1}
From the proof of Case 3, we can see that when
$\rho({X_\textrm{min}})$ increases, the upper bound of
$\textrm{Vol}(-(K_X+\Delta))$ decreases. More precisely:

1. when $\rho({X_\textrm{min}})=1$,  ${X_\textrm{min}} \simeq
\mathbb{P}^2$, and $\textrm{Vol}(-(K_X+\Delta))\leqslant 9$;

2. when $\rho({X_\textrm{min}})=2$,  ${X_\textrm{min}} \simeq
\mathbb{F}_n$ for some $n\leqslant \lfloor 2/\epsilon \rfloor$, and
$\textrm{Vol}(-(K_X+\Delta))\leqslant \lfloor 2/\epsilon \rfloor+
4+\frac{4}{\lfloor 2/\epsilon \rfloor}$;

3. when $\rho({X_\textrm{min}})\geqslant 3$, there exists a
birational morphism ${X_\textrm{min}} \rightarrow \mathbb{F}_n$ for
some $n\leqslant \lfloor 2/\epsilon \rfloor$, and
$\textrm{Vol}(-(K_X+\Delta)) < \lfloor 2/\epsilon \rfloor+
4+\frac{4}{\lfloor 2/\epsilon \rfloor}$.

Hence we want to decide the optimal bound of the volume up to
$\rho({X_\textrm{min}})$. In the following two sections, we give the
optimal bounds for the cases when $\rho({X_\textrm{min}})\geqslant
3$ and $\rho({X_\textrm{min}})\geqslant 4$.
\end{remark}

\section{Bounding the volumes II: $\rho({X_\textrm{min}})\geqslant 3$}

In this section we prove Theorem \ref{r3}.

By taking minimal resolution, we may assume that $X$ is smooth  and
$\Delta$ is with coefficients in $[0, 1-\epsilon]$.
By Remark \ref{rem1}, we may assume that $\rho(X)=3$.
When $\rho(X)=3$, by Lemma \ref{rat lem}, $X$ is given by blowing up
one point $p$ on $\mathbb{F}_n$  for some $n\leqslant \lfloor
2/\epsilon \rfloor$. By Proposition \ref{blowupposition}, we may
assume that $p \not\in S_n$.
Let $\phi: X\rightarrow \mathbb{F}_n$ be the blow-up at $p\in
\mathbb{F}_n$, $E\subset X$ be the exceptional curve and
$\overline{\Delta}$ be the push-forward of $\Delta$. Write
$\overline{\Delta}=aS_n+\sum d_i\overline{\Delta}_i$, where
$\overline{\Delta}_i$ is an irreducible curve different from $S_n$
and $\overline{\Delta}_i\equiv \alpha_i h+\beta_i f$ with
$\alpha_i,\beta_i \geqslant 0$.
Then by (\ref{star}), we have
$$
K_X+\Delta=\phi^*(K_{\mathbb{F}_n}+\overline{\Delta})+eE
$$
with $e\geqslant 0$. Hence
$$
\Delta=aS_n^\prime+\sum
d_i\Delta_i+(\mbox{mult}_p(\overline{\Delta})-1+e)E,
$$
where $S_n^\prime$ (resp. $\Delta_i$) is the strict transform of
$S_n$ (resp. $\overline{\Delta}_i$) on $X$. Hence in this setting
\begin{align}\label{mult4}
\mbox{mult}_p(\overline{\Delta})-1+e\geqslant 0.
\end{align}
Since $K_X+\Delta$ is anti-nef, so is
$K_{\mathbb{F}_n}+\overline{\Delta}$. Hence
\begin{align}\label{4.1}
0\geqslant (K_{\mathbb{F}_n}+\overline{\Delta}).S_n=n-2-na+\sum d_i
\beta_i;
\end{align}
\begin{align}\label{4.2}
0\geqslant (K_{\mathbb{F}_n}+\overline{\Delta}).f= -2+a+\sum d_i
\alpha_i .
\end{align}
Now denote the fiber in $\mathbb{F}_n$ passing through $p$ as $F_p$.
And we may assume that $\overline{\Delta}_1=F_p$.
Then  $\sum d_i\alpha_i$ can be estimated by the multiplicity
inequality (\ref{mult4}). We have
\begin{align*}
{}& \sum d_i\alpha_i \\
= {}& \left(\sum d_i \overline{\Delta}_i\right).F_p
=  \Big(\sum_{i>1} d_i \overline{\Delta}_i\Big).F_p\\
\geqslant {}& \mbox{mult}_p\Big(\sum_{i>1} d_i
\overline{\Delta}_i\Big)
=  \mbox{mult}_p\left(\sum d_i \overline{\Delta}_i\right)-d_1\\
= {}& \mbox{mult}_p(\overline{\Delta})-d_1\\
\geqslant {}& 1-e-d_1.
\end{align*}
Hence
\begin{align}\label{a4}
\sum d_i\alpha_i \geqslant \max\{0, 1-e-d_1\},
\end{align}
and
\begin{align}\label{b4}
\sum d_i\beta_i \geqslant d_1\beta_1=d_1.
\end{align}
Now we can estimate the volume. We have
\begin{align*}
{}&(K_X+\Delta)^2 \\
= {}& (K_{\mathbb{F}_n}+\overline{\Delta})^2-e^2\\
= {}& n\left(-2+a+\sum d_i\alpha_i\right)^2+2\left(-2+a+\sum
d_i\alpha_i\right)\left(n-2-na+\sum
d_i\beta_i\right)-e^2\\
\leqslant {}& n(-2+a+\max\{0, 1-e-d_1\})^2\\
{}& +2(-2+a+\max\{0,
1-e-d_1\})(n-2-na+d_1)-e^2\\
\leqslant {}& n(-2+a )^2+2(-2+a )(n-2-na+d_1)-(1-d_1)^2\\
\leqslant {}& n(-2+a )^2+2(-2+a )(n-2-na  )- 1 \\
\leqslant {}&
\begin{cases}
n+3+\frac{4}{n} & \mbox{if } n\geqslant 2\\
7  & \mbox{if } n=0,1
\end{cases}\\
\leqslant {}& \lfloor 2/\epsilon \rfloor+ 3+\frac{4}{\lfloor
2/\epsilon \rfloor}.
\end{align*}

The first inequality follows from (\ref{4.1}), (\ref{4.2}),  (\ref{a4})
and (\ref{b4}).
The second inequality holds since  the target can be viewed as a
function of $e$ and it reaches the maximum when $e=1-d_1$.
The third inequality holds since the target  can be viewed  as a
function of $d_1$ and it reaches the maximum when $d_1=0$.
The forth inequality holds since the target  can be viewed  as a
function of $a$ and it reaches the maximum when $a=\max\{0,
1-\frac{2}{n}\}$.
The last inequality holds since the target function is increasing
and $n\leqslant \lfloor 2/\epsilon \rfloor$.

Finally we can see that the bound is optimal by an example.
\begin{example}
Take $X$ to be the blow-up of $\mathbb{F}_n$ at $p$, where
$n=\lfloor 2/\epsilon \rfloor$ and $p\not\in S_n$. Then $(X,
(1-\frac{2}{n})S_n^\prime)$ is a weak log del Pezzo surface (cf.
Section 7, Theorem \ref{example2}(2)), and
$$
\Big(K_X+\Big(1-\frac{2}{n}\Big)S_n^\prime\Big)^2=\lfloor 2/\epsilon
\rfloor+ 3+\frac{4}{\lfloor 2/\epsilon \rfloor},
$$
where $S_n^\prime$ is the strict transform of $S_n$ on $X$.
\end{example}

\section{Bounding the volumes III: $\rho({X_\textrm{min}})\geqslant 4$}

In this section we prove Theorem \ref{r4}.

By taking minimal resolution, we may assume that $X$ is smooth  and
$\Delta$ is with coefficients in $[0, 1-\epsilon]$.
By Remark \ref{rem1}, we may assume that $\rho(X)=4$.
When $\rho(X)=4$, by Lemma \ref{rat lem}, $X$ is given by blowing up
one point $p_2$ on some $X_1$ while  $X_1$ is given by blowing up
one point $p_1$ on $\mathbb{F}_n$ for some $n\leqslant \lfloor
2/\epsilon \rfloor$. By Proposition \ref{blowupposition}, we may
assume that $p_1\not\in S_n$ and $p_2\not \in S_n^\prime$ where
$S_n^\prime$ is the strict transform of $S_n$ on $X_1$.
Let $\phi_1: X_1\rightarrow \mathbb{F}_n$ and $\phi_2: X\rightarrow
X_1$ be the blow-ups, $E_2\subset X$ and $E_1 \subset X_1$ be the
exceptional curves, $\Delta^\prime$ and $\overline{\Delta}$ be the
push-forwards of $\Delta$ by $\phi_2$ and $\phi_1\circ\phi_2$
respectively. Write $\overline{\Delta}=aS_n+\sum
d_i\overline{\Delta}_i$ where $\overline{\Delta}_i$ is an
irreducible curve different from $S_n$ and
$\overline{\Delta}_i\equiv \alpha_i h+\beta_i f$ with
$\alpha_i,\beta_i \geqslant 0$.
Then by (\ref{star}), we have
$$
K_{X_1}+\Delta^\prime=\phi_1^*(K_{\mathbb{F}_n}+\overline{\Delta})+e_1E_1
$$
and
$$
K_{X}+\Delta=\phi_2^*(K_{X_1}+\Delta^\prime)+e_2E_2
$$
with $e_1, e_2\geqslant 0$. Hence
$$
\Delta^\prime=aS_n^\prime+\sum
d_i\Delta_i^\prime+(\mbox{mult}_{p_1}(\overline{\Delta})-1+e_1)E_1,
$$
where $\Delta_i^\prime$ is the strict transform of
$\overline{\Delta}_i$ on $X_1$. Hence in this setting
\begin{align}\label{mult51}
\mbox{mult}_{p_1}(\overline{\Delta})-1+e_1\geqslant 0.
\end{align}
Similarly, we have
\begin{align}\label{mult52}
\mbox{mult}_{p_2}(\Delta^\prime)-1+e_2\geqslant 0.
\end{align}
Since $K_X+\Delta$ is anti-nef, so is
$K_{\mathbb{F}_n}+\overline{\Delta}$. Hence
\begin{align}\label{51}
0\geqslant (K_{\mathbb{F}_n}+\overline{\Delta}).S_n=n-2-na+\sum d_i
\beta_i ;
\end{align}
\begin{align}
\label{52} 0\geqslant
(K_{\mathbb{F}_n}+\overline{\Delta}).f=-2+a+\sum d_i \alpha_i.
\end{align}

Now there are two cases for the position of $p_2$: $p_2 \not\in E_1$
or $p_2 \in E_1$.

\textbf{Case 1: $p_2 \not\in E_1$. }

In this case, $p_2$ can be viewed as a point on $\mathbb{F}_n$, so
denote the fiber  of $\mathbb{F}_n \rightarrow \mathbb{P}^1$ passing through it by $F_{p_2}$. Denote the fiber
passing through $p_1$ by $F_{p_1}$.

There are two subcases: $F_{p_1}\neq F_{p_2}$ or $F_{p_1}= F_{p_2}$.

\textbf{Subcase 1.1: $F_{p_1}\neq F_{p_2}$. }

In this subcase, we may assume that $\overline{\Delta}_i=F_{p_i}$
for $i=1,2$.
Then $\sum d_i\alpha_i$ can be estimated by the multiplicity
inequalities (\ref{mult51}) and (\ref{mult52}). We have
\begin{align*}
 {}&\sum d_i\alpha_i \\
= {}& \left(\sum d_i \overline{\Delta}_i\right).F_{p_1}
=  \Big(\sum_{i\neq 1} d_i \overline{\Delta}_i\Big).F_{p_1}\\
\geqslant {}& \mbox{mult}_{p_1}\Big(\sum_{i\neq 1} d_i
\overline{\Delta}_i\Big)
=\mbox{mult}_{p_1}\left(\sum d_i \overline{\Delta}_i\right)-d_1\\
= {}& \mbox{mult}_{p_1}(\overline{\Delta})-d_1\\
\geqslant {}& 1-e_1-d_1,
\end{align*}
and also
\begin{align*}
{}&\sum d_i\alpha_i \\
= {}& \left(\sum d_i \overline{\Delta}_i\right).F_{p_2}
= \Big(\sum_{i\neq 2} d_i \overline{\Delta}_i\Big).F_{p_2}\\
\geqslant {}& \mbox{mult}_{p_2}\Big(\sum_{i\neq 2} d_i
\overline{\Delta}_i\Big)
=  \mbox{mult}_{p_2}\left(\sum d_i \overline{\Delta}_i\right)-d_2\\
= {}& \mbox{mult}_{p_2}(\overline{\Delta})-d_2
=  \mbox{mult}_{p_2}(\Delta^\prime)-d_2\\
\geqslant {}& 1-e_2-d_2.
\end{align*}
Hence
\begin{align}\label{511a}
\sum d_i\alpha_i \geqslant \max\{0,1-e_1-d_1,1-e_2-d_2\},
\end{align}
and
\begin{align}\label{511b}
\sum d_i\beta_i \geqslant d_1\beta_1+d_2\beta_2= d_1+d_2.
\end{align}
Now we can estimate the volume. We have
\begin{align*}
{}&(K_X+\Delta)^2\\
= {}& (K_{\mathbb{F}_n}+\overline{\Delta})^2-e_1^2-e_2^2\\
= {}& n\left(-2+a+\sum d_i\alpha_i\right)^2+2\left(-2+a+\sum
d_i\alpha_i\right)\left(n-2-na+\sum
d_i\beta_i\right)-e_1^2-e_2^2\\
\leqslant {}& n(-2+a+\max\{0,1-e_1-d_1,1-e_2-d_2\})^2\\
   {}&+2(-2+a+\max\{0,1-e_1-d_1,1-e_2-d_2\})(n-2-na+d_1+d_2)-e_1^2-e_2^2\\
\leqslant {}& n(-2+a )^2+2(-2+a )(n-2-na+d_1+d_2)-(1-d_1)^2-(1-d_2)^2\\
\leqslant {}& n(-2+a )^2+2(-2+a )(n-2-na  )- 2 \\
\leqslant {}&
\begin{cases}
n+2+\frac{4}{n} & \mbox{if } n\geqslant 2\\
6  & \mbox{if } n=0,1
\end{cases}\\
\leqslant {}& \lfloor 2/\epsilon \rfloor+ 2+\frac{4}{\lfloor
2/\epsilon \rfloor}.
\end{align*}

The first inequality follows from (\ref{51}), (\ref{52}), (\ref{511a})
and (\ref{511b}).
The second inequality holds since the target can be viewed  as a
function of $e_1$ and $e_2$ and it reaches the maximum when
$e_1=1-d_1$ and $e_2=1-d_2$.
The third inequality holds since the target can be viewed  as a
function of $d_1$ and $d_2$ and it reaches the maximum when
$d_1=d_2=0$.
The forth inequality holds since the target  can be viewed  as a
function of $a$ and it reaches the maximum when
$a=\max\{0,1-\frac{2}{n}\}$.
The last inequality holds since the target function is increasing
and $n\leqslant \lfloor 2/\epsilon \rfloor$ .

\textbf{Subcase 1.2: $F_{p_1}= F_{p_2}$. }

In this subcase, we may assume that
$\overline{\Delta}_1=F_{p_1}=F_{p_2}=:F_p$.
Then  $\sum d_i\alpha_i$ can be estimated by the multiplicity
inequalities (\ref{mult51}) and (\ref{mult52}). We have
\begin{align*}
{}&\sum d_i\alpha_i\\
= {}& \left(\sum d_i \overline{\Delta}_i\right).F_p
=   \Big(\sum_{i\neq 1} d_i \overline{\Delta}_i\Big).F_{p}\\
\geqslant {}& \mbox{mult}_{p_1}\Big(\sum_{i\neq 1} d_i \overline{\Delta}_i\Big)+\mbox{mult}_{p_2}\Big(\sum_{i\neq 1} d_i \overline{\Delta}_i\Big)\\
= {}& \mbox{mult}_{p_1}\left(\sum d_i \overline{\Delta}_i\right)+\mbox{mult}_{p_2}\left(\sum d_i \overline{\Delta}_i\right)-2d_1\\
= {}& \mbox{mult}_{p_1}(\overline{\Delta})+\mbox{mult}_{p_2}(\overline{\Delta})-2d_1\\
= {}&\mbox{mult}_{p_1}(\overline{\Delta})+\mbox{mult}_{p_2}(\Delta^\prime)-2d_1\\
\geqslant {}& 2-e_1-e_2-2d_1.
\end{align*}
Hence
\begin{align}\label{512a}
\sum d_i\alpha_i \geqslant \max\{0,2-e_1-e_2-2d_1\},
\end{align}
and
\begin{align}\label{512b}
\sum d_i\beta_i \geqslant d_1\beta_1=d_1 .
\end{align}
Since $K_X+\Delta$ is anti-nef,
\begin{align}\label{5121}
0\geqslant (K_X+\Delta).S_n^{\prime\prime}=
(K_{\mathbb{F}_n}+\overline{\Delta}).S_n=n-2-na+\sum d_i
\beta_i\geqslant n-2-na+d_1;
\end{align}
\begin{align}\label{5122}
0\geqslant (K_X+\Delta).F_p^{\prime\prime}=
(K_{\mathbb{F}_n}+\overline{\Delta}).F_p+e_1+e_2=-2+a+\sum d_i
\alpha_i+e_1+e_2\geqslant a-2d_1,
\end{align}
where $S_n^{\prime\prime}$ (resp. $F_p^{\prime\prime}$) is the
strict transform of $S_n$ (resp. $F_p$) on $X$. The second
inequalities of (\ref{5121}) and (\ref{5122}) follow from
(\ref{512b}) and (\ref{512a}) respectively. Hence
$$
\frac{n-2+d_1}{n}\leqslant a \leqslant 2d_1.
$$
This implies
$$
d_1\geqslant \frac{n-2}{2n-1},
$$
and
$$
a\geqslant \frac{2n-4}{2n-1}.
$$
Also by assumption $a\leqslant 1-\epsilon$, hence
\begin{align}\label{512n}
n\leqslant \lfloor(3+\epsilon)/2\epsilon \rfloor.
\end{align}
Now we can estimate the volume. We have
\begin{align*}
{}&(K_X+\Delta)^2\\
= {}& (K_{\mathbb{F}_n}+\overline{\Delta})^2-e_1^2-e_2^2\\
= {}& n\left(-2+a+\sum d_i\alpha_i\right)^2+2\left(-2+a+\sum
d_i\alpha_i\right)\left(n-2-na+\sum
d_i\beta_i\right)-e_1^2-e_2^2\\
\leqslant {}& n(-2+a+\max\{0,2-e_1-e_2-2d_1\})^2\\
 {}&+2(-2+a+\max\{0,2-e_1-e_2-2d_1\})(n-2-na+d_1)-e_1^2-e_2^2\\
\leqslant {}& n(-2+a )^2+2(-2+a )(n-2-na+d_1)-2(1-d_1)^2\\
\leqslant {}& n(-2+a )^2+2(-2+a )(n-2-na+\frac{a}{2})-2(1-\frac{a}{2})^2 \\
\leqslant {}&
\begin{cases}
n+\frac{5}{2}+\frac{9}{4n-2} & \mbox{if } n\geqslant 2\\
6  & \mbox{if } n=0,1
\end{cases}\\
\leqslant {}& \lfloor (3+\epsilon)/2\epsilon
\rfloor+\frac{5}{2}+\frac{9}{4\lfloor (3+\epsilon)/2\epsilon
\rfloor-2}.
\end{align*}

The first inequality follows from (\ref{51}), (\ref{52}), (\ref{512a})
and (\ref{512b}).
The second inequality holds since the target can be viewed  as a
function of $e_1$ and $e_2$ and it reaches the maximum when
$e_1=e_2=1-d_1$.
The third inequality holds since the target  can be viewed  as a
function of $d_1$ and it reaches the maximum when $d_1=\frac{a}{2}$.
The forth inequality holds since the target  can be viewed  as a
function of $a$ and it reaches the maximum when
$a=\max\{0,\frac{2n-4}{2n-1}\}$.
The last inequality holds since the target function is increasing
and $n\leqslant \lfloor(3+\epsilon)/2\epsilon \rfloor$ by
(\ref{512n}).

\begin{remark}
When $\epsilon$ is small enough, it is easy to see that
$$
\lfloor(3+\epsilon)/2\epsilon \rfloor+\frac{5}{2}+\frac{9}{4\lfloor
(3+\epsilon)/2\epsilon \rfloor-2}< \lfloor 2/\epsilon \rfloor+
2+\frac{4}{\lfloor 2/\epsilon \rfloor},
$$
but in general the inequality does not always hold, for example when
$\epsilon=\frac{3}{5}$.
\end{remark}

\textbf{Case 2: $p_2\in E_1$. }

There are two subcases: $p_2\in E_1\cap F_{p_1}^\prime$ or $p_2\in
E_1 \backslash F_{p_1}^\prime$, where $F_{p_1}^\prime$ is the strict
transform of $F_{p_1}$ on $X_1$.

\textbf{Subcase 2.1: $p_2\in E_1\cap F_{p_1}^\prime$. }

In this subcase, we may assume that $\overline{\Delta}_1=F_{p_1}$.
Then  $\sum d_i\alpha_i$ can be estimated by the multiplicity
inequalities (\ref{mult51}) and (\ref{mult52}). We have
\begin{align*}
{}&\sum d_i\alpha_i \\
= {}& \left(\sum d_i \overline{\Delta}_i\right).F_{p_1} =
\Big(\sum_{i\neq 1} d_i \overline{\Delta}_i\Big).F_{p_1}
= \phi_1^*\Big(\sum_{i\neq 1} d_i \overline{\Delta}_i\Big).F_{p_1}^\prime\\
= {}& \Big(\sum_{i\neq 1} d_i \Delta_i^\prime+\mbox{mult}_{p_1}\Big(\sum_{i\neq 1} d_i \overline{\Delta}_i\Big)E_1\Big).F_{p_1}^\prime\\
\geqslant {}& \mbox{mult}_{p_2}\Big(\sum_{i\neq 1} d_i \Delta_i^\prime\Big)+\mbox{mult}_{p_1}\Big(\sum_{i\neq 1} d_i \overline{\Delta}_i\Big)\\
= {}& \mbox{mult}_{p_2}\left(\sum  d_i \Delta_i^\prime\right)+\mbox{mult}_{p_1}\left(\sum  d_i \overline{\Delta}_i\right)-2d_1\\
= {}& \mbox{mult}_{p_2}\left(\sum  d_i \Delta_i^\prime\right)+\left(\mbox{mult}_{p_1}\left(\sum  d_i \overline{\Delta}_i\right)-1+e_1\right)-2d_1+1-e_1\\
= {}& \mbox{mult}_{p_2}\left(\sum  d_i \Delta_i^\prime+\left(\mbox{mult}_{p_1}\left(\sum  d_i \overline{\Delta}_i\right)-1+e_1\right)E_1\right)-2d_1+1-e_1\\
= {}& \mbox{mult}_{p_2}( \Delta^\prime )-2d_1+1-e_1\\
\geqslant {}& (1-e_2 )-2d_1+1-e_1\\
= {}& 2-e_1-e_2-2d_1.
\end{align*}
The rest follows exactly as in Subcase 1.2.

\textbf{Subcase 2.2: $p_2\in E_1 \backslash F_{p_1}^\prime$. }

Then $\sum d_i\alpha_i$ can be estimated by the multiplicity
inequalities (\ref{mult51}) and (\ref{mult52}). We have
\begin{align*}
 {}&\sum d_i\alpha_i \\
= {}& \left(\sum d_i \overline{\Delta}_i\right).F_{p_1} =
\Big(\sum_{i\neq 1} d_i \overline{\Delta}_i\Big).F_{p_1}
\geqslant  \mbox{mult}_{p_1}\Big(\sum_{i\neq 1} d_i \overline{\Delta}_i\Big)\\
= {}& \mbox{mult}_{p_1}(\overline{\Delta})-d_1,
\end{align*}
and
\begin{align*}
{}&\sum d_i\alpha_i\\
= {}&\left(\sum d_i \overline{\Delta}_i\right).F_{p_1}
=  \Big(\sum_{i\neq 1} d_i \overline{\Delta}_i\Big).F_{p_1}\\
= {}& \Big(\sum_{i\neq 1} d_i \Delta_i^\prime\Big).\phi_1^*F_{p_1}
=  \Big(\sum_{i\neq 1} d_i \Delta_i^\prime\Big).(F_{p_1}^\prime+E_1)\\
\geqslant {}& \Big(\sum_{i\neq 1} d_i \Delta_i^\prime\Big). E_1
\geqslant \mbox{mult}_{p_2}\Big(\sum_{i\neq 1} d_i \Delta_i^\prime\Big)\\
\geqslant {}& \mbox{mult}_{p_2}\left(\sum d_i \Delta_i^\prime+(\mbox{mult}_{p_1}(\overline{\Delta})-1+e_1)E_1\right)-(\mbox{mult}_{p_1}(\overline{\Delta})-1+e_1)\\
= {}& \mbox{mult}_{p_2}(\Delta^\prime)-(\mbox{mult}_{p_1}(\overline{\Delta})-1+e_1)\\
\geqslant {}& 2-e_2-e_1-\mbox{mult}_{p_1}(\overline{\Delta}).
\end{align*}
Combining the above two inequalities together, we have
\begin{align*}
{}&\sum d_i\alpha_i \\
\geqslant {}&\frac{1}{2}(\mbox{mult}_{p_1}(\overline{\Delta})-d_1)+
\frac{1}{2}(2-e_2-e_1-\mbox{mult}_{p_1}(\overline{\Delta}))\\
= {}& \frac{2-e_1-e_2-d_1}{2}.
\end{align*}
Hence
\begin{align}\label{522a}
\sum d_i\alpha_i \geqslant \max\left\{0,
\frac{2-e_1-e_2-d_1}{2}\right\},
\end{align}
and
\begin{align}\label{522b}
\sum d_i\beta_i \geqslant d_1\beta_1=d_1.
\end{align}
Now we can estimate the volume. We have
\begin{align*}
{}&(K_X+\Delta)^2 \\
= {}& (K_{\mathbb{F}_n}+\overline{\Delta})^2-e_1^2-e_2^2\\
= {}& n\left(-2+a+\sum d_i\alpha_i\right)^2+2\left(-2+a+\sum
d_i\alpha_i\right)\left(n-2-na+\sum
d_i\beta_i\right)-e_1^2-e_2^2\\
\leqslant {}& n\Big(-2+a+\max\left\{0, \frac{2-e_1-e_2-d_1}{2}\right\}\Big)^2\\
{}&+2\Big(-2+a+\max\left\{0, \frac{2-e_1-e_2-d_1}{2}\right\}\Big)(n-2-na+d_1)-e_1^2-e_2^2\\
\leqslant {}& n(-2+a )^2+2(-2+a )(n-2-na+d_1)-2(1-\frac{d_1}{2})^2\\
\leqslant {}& n(-2+a )^2+2(-2+a )(n-2-na  )- 2 \\
\leqslant {}&\begin{cases}
n+2+\frac{4}{n} & \mbox{if } n\geqslant 2\\
6  & \mbox{if } n=0,1
\end{cases}\\
\leqslant {}& \lfloor 2/\epsilon \rfloor+ 2+\frac{4}{\lfloor
2/\epsilon \rfloor}.
\end{align*}

The first inequality follows from (\ref{51}), (\ref{52}), (\ref{522a})
and (\ref{522b}).
The second inequality holds since  the target  can be viewed  as a
function of $e_1$ and $e_2$ and it reaches the maximum when
$e_1=e_2=1-\frac{d_1}{2}$.
The third inequality holds since the target  can be viewed  as a
function of $d_1$ and it reaches the maximum when $d_1= 0$.
The forth inequality holds since the target  can be viewed  as a
function of $a$ and it reaches the maximum when
$a=\max\{0,1-\frac{2}{n}\}$.
The last inequality holds since the target function is increasing
and $n\leqslant \lfloor 2/\epsilon \rfloor$.

Finally we can see that the bound is optimal by two examples.

\begin{example}
Take $X$ to be the blow-up of $\mathbb{F}_n$ at $p_1$ and $p_2$,
where $n=\lfloor 2/\epsilon \rfloor$, $p_1, p_2\not\in S_n$ and
$F_{p_1}\neq F_{p_2}$. Then $(X, (1-\frac{2}{n})S_n^\prime)$ is a
weak log del Pezzo surface (cf. Section 7, Theorem
\ref{example2}(2)), and
$$
\Big(K_X+\Big(1-\frac{2}{n}\Big)S_n^\prime\Big)^2=\lfloor 2/\epsilon
\rfloor+ 2+\frac{4}{\lfloor 2/\epsilon \rfloor},
$$
where $S_n^\prime$ is the strict transform of $S_n$ on $X$.
\end{example}
\begin{example}
Take $X$ to be the blow-up of $\mathbb{F}_n$ at $p_1$ and $p_2$,
where $n =\lfloor (3+\epsilon)/2\epsilon \rfloor$, $p_1, p_2\not\in
S_{n}$ and $F_{p_1}=F_{p_2}=F$. Then $(X,
\frac{2n-4}{2n-1}S_{n}^\prime+\frac{ n-2}{2n-1}F^\prime)$ is a weak
log del Pezzo surface (cf. Section 7, Theorem \ref{example1}), and
$$
\Big(K_X+\frac{2n-4}{2n-1}S_{n}^\prime+\frac{
n-2}{2n-1}F^\prime\Big)^2=\lfloor (3+\epsilon)/2\epsilon
\rfloor+\frac{5}{2}+\frac{9}{4\lfloor (3+\epsilon)/2\epsilon
\rfloor-2},
$$
where $S_{n}^\prime$ (resp. $F^\prime$) is the strict transform of
$S_{n}$ (resp. $F$) on $X$.
\end{example}

\begin{remark}
By this method, we can calculate the optimal bounds when
$\rho({X_\textrm{min}})$ gets higher, but it is getting more and
more complicated.
\end{remark}

\section{Blowing up points on $\mathbb{F}_n$ in general position}
By Lemma \ref{rat lem}, almost all smooth weak log del Pezzo
surfaces come from blowing up the Hirzebruch surfaces. So we
consider the case when blowing up the Hirzebruch surfaces at points
 in general position.
In this section we are going to prove the following theorem. It is
actually a generalization of Section 5, Subcase 1.1.
\begin{thm}\label{general blowup}
Let $X$ be a smooth surface which is given by blowing up $k$ points
on $\mathbb{F}_n$. Assume that these points are not on $S_n$, and no
two of them are on the same fiber  of $\mathbb{F}_n \rightarrow \mathbb{P}^1$. Assume that $-(K_X+\Delta)$ is
nef and big, where $\Delta$ is an effective $\mathbb{R}$-divisor. Then the anti-canonical volume
${\rm Vol}(-(K_X+\Delta))=(K_X+\Delta)^2$ satisfies
$$
(K_X+\Delta)^2\leqslant\begin{cases}
n+4+\frac{4}{n}-k & \mbox{if } n\geqslant 2;\\
8-k  & \mbox{if } n=0,1.
\end{cases}
$$
\end{thm}

\begin{remark}
By Theorem \ref{example2}, the bound given above is optimal for
points in general position.
\end{remark}

Since the volume of a nef and big divisor is always positive, we can
tell when such $X$ is of weak log del Pezzo type.

\begin{cor}\label{cor}
Let $X$ be a smooth surface which is given by blowing up at $k$
points on $\mathbb{F}_n$. Assume that these points are not on $S_n$,
and no two of them are on the same fiber  of $\mathbb{F}_n \rightarrow \mathbb{P}^1$.  Assume that there exists
an effective $\mathbb{R}$-divisor $\Delta$ on $X$ such that $-(K_X+\Delta)$ is nef and big.
Then
$$
k\leqslant\begin{cases}
n+4 & \mbox{if } n\geqslant 4;\\
n+5  & \mbox{if } n=2,3;\\
7  & \mbox{if } n=0,1.
\end{cases}
$$
\end{cor}

\begin{proof}[Proof of Theorem \ref{general blowup}]
Let $\phi: X\rightarrow \mathbb{F}_n$ be the blow-up at $p_1, p_2,
\ldots, p_k \in \mathbb{F}_n$ and $E_1, E_2, \ldots, E_k\subset X$
be the exceptional curves and $\overline{\Delta}$ be the
push-forward of $\Delta$. Write $\overline{\Delta}=aS_n+\sum
d_i\overline{\Delta}_i$ where $\overline{\Delta}_i$ is an
irreducible curve different from $S_n$ and
$\overline{\Delta}_i\equiv \alpha_i h+\beta_i f$ with
$\alpha_i,\beta_i \geqslant 0$.
Then by (\ref{star}), we have
$$
K_X+\Delta=\phi^*(K_{\mathbb{F}_n}+\overline{\Delta})+\sum_{j=1}^ke_jE
_j$$ with $e_j\geqslant 0$ for $1\leqslant j \leqslant k$. Hence
$$
\Delta=aS_n^\prime+\sum
d_i\Delta_i+\sum_{j=1}^k(\mbox{mult}_{p_j}(\overline{\Delta})-1+e_j)E_j,
$$
where $S_n^\prime$ (resp. $\Delta_i$) is the strict transform of
$S_n$ (resp. $\overline{\Delta}_i$) on $X$. Hence in this setting
\begin{align}\label{mult6}
\mbox{mult}_{p_j}(\overline{\Delta})-1+e_j\geqslant 0
\end{align}
for $1\leqslant j \leqslant k$.
Since $K_X+\Delta$ is anti-nef, so is
$K_{\mathbb{F}_n}+\overline{\Delta}$. Hence
\begin{align}\label{61}
0\geqslant (K_{\mathbb{F}_n}+\overline{\Delta}).S_n=n-2-na+\sum d_i
\beta_i  ;
\end{align}
\begin{align}\label{62}
0\geqslant (K_{\mathbb{F}_n}+\overline{\Delta}).f= -2+a+\sum d_i
\alpha_i  .
\end{align}
Now denote the fiber in $\mathbb{F}_n$ passing through $p_j$ as
$F_{p_j}$  for $1\leqslant j \leqslant k$. And we may assume that
$\overline{\Delta}_j=F_{p_j}$ for $1\leqslant j \leqslant k$.
Then for $1\leqslant j \leqslant k$,  $\sum d_i\alpha_i$ can be
estimated by the multiplicity inequality (\ref{mult6}). We have
\begin{align*}
{}&\sum d_i\alpha_i \\
= {}& \left(\sum d_i \overline{\Delta}_i\right).F_{p_j}
=  \Big(\sum_{i\neq j} d_i \overline{\Delta}_i\Big).F_{p_j}\\
\geqslant {}& \mbox{mult}_{p_j}\Big(\sum_{i\neq j} d_i
\overline{\Delta}_i\Big)
=  \mbox{mult}_{p_j}\left(\sum d_i \overline{\Delta}_i\right)-d_j\\
= {}& \mbox{mult}_{p_j}(\overline{\Delta})-d_j\\
\geqslant {}& 1-e_j-d_j.
\end{align*}
Hence
\begin{align}\label{6a}
\sum d_i\alpha_i \geqslant \max_{1\leqslant j \leqslant k}\{0,
1-e_j-d_j\},
\end{align}
and
\begin{align}\label{6b}
\sum d_i\beta_i \geqslant \sum_{j=1}^k d_j\beta_j=\sum_{j=1}^k d_j.
\end{align}
Now we can estimate the volume. We have
\begin{align*}
{}&(K_X+\Delta)^2 \\
= {}& (K_{\mathbb{F}_n}+\overline{\Delta})^2-\sum_{j=1}^ke_j^2\\
= {}& n\left(-2+a+\sum d_i\alpha_i\right)^2+2\left(-2+a+\sum
d_i\alpha_i\right)\left(n-2-na+\sum
d_i\beta_i\right)-\sum_{j=1}^ke_j^2\\
\leqslant {}& n(-2+a+\max_{1\leqslant j \leqslant k}\{0,
1-e_j-d_j\})^2\\
{}&+2(-2+a+\max_{1\leqslant j \leqslant k}\{0, 1-e_j-d_j\})\Big(n-2-na+\sum_{j=1}^k d_j\Big)-\sum_{j=1}^ke_j^2\\
\leqslant {}& n(-2+a )^2+2(-2+a )\Big(n-2-na+\sum_{j=1}^k d_j\Big)-\sum_{j=1}^k(1-d_j)^2\\
\leqslant {}& n(-2+a )^2+2(-2+a )(n-2-na  )- k\\
\leqslant {}& \begin{cases}
n+4+\frac{4}{n}-k & \mbox{if } n\geqslant 2;\\
8-k  & \mbox{if } n=0,1.
\end{cases}
\end{align*}

The first inequality follows from (\ref{61}), (\ref{62}), (\ref{6a}) and
(\ref{6b}).
The second inequality holds since the target  can be viewed  as a
function of $e_j$ for all $1\leqslant j \leqslant k$ and it reaches
the maximum when $e_j= 1-d_j$ for all $1\leqslant j \leqslant k$.
The third inequality holds since the target  can be viewed  as a
function of $d_j$ for all $1\leqslant j \leqslant k$ and it reaches
the maximum when $d_j= 0$  for all $1\leqslant j \leqslant k$.
The last inequality holds since the target  can be viewed  as a
function of $a$ and it reaches the maximum when
$a=\max\{0,1-\frac{2}{n}\}$.
\end{proof}

\section{Examples}
In this section we give some examples of smooth weak log del Pezzo
surfaces which make the above inequalities optimal.

\begin{thm}\label{example1}
Let $X$ be the blow-up of $\mathbb{F}_n$ at $p_1$ and $p_2$ for some
$n\geqslant2$, where
 $p_1, p_2\not\in S_{n}$ and $F_{p_1}=F_{p_2}=F$. Then $(X,
\frac{2n-4}{2n-1}S_{n}^\prime+\frac{ n-2}{2n-1}F^\prime)$ is a weak
log del Pezzo surface, where $S_{n}^\prime$ (resp. $F^\prime$) is
the strict transform of $S_{n}$ (resp. $F$) on $X$.
\end{thm}

\begin{proof}
The only thing to prove is the anti-nefness of
$K_X+\frac{2n-4}{2n-1}S_{n}^\prime+\frac{ n-2}{2n-1}F^\prime$.

Take an irreducible curve $C^\prime$ in $X$, which is not
exceptional over $\mathbb{F}_n$, and take $C$ to be the push-forward
of $C^\prime$. Let $\phi$ be the blow-up. Then
\begin{align*}
{}&\Big(K_X+\frac{2n-4}{2n-1}S_{n}^\prime+\frac{ n-2}{2n-1}F^\prime\Big).C^\prime\\
= {}&\Big(\phi^*\Big(K_{\mathbb{F}_n}+\frac{2n-4}{2n-1}S_{n}+\frac{
n-2}{2n-1}F\Big)+
\frac{n+1}{2n-1}E_1+\frac{n+1}{2n-1}E_2\Big).C^\prime\\
= {}&\Big(K_{\mathbb{F}_n}+\frac{2n-4}{2n-1}S_{n}+\frac{
n-2}{2n-1}F\Big).C+\frac{n+1}{2n-1}\mbox{mult}_{p_1}C+\frac{n+1}{2n-1}\mbox{mult}_{p_2}C\\
={}&-\frac{2n+2}{2n-1}h.C+\frac{n+1}{2n-1}\mbox{mult}_{p_1}C+\frac{n+1}{2n-1}\mbox{mult}_{p_2}C.
\end{align*}
Hence $K_X+\frac{2n-4}{2n-1}S_{n}^\prime+\frac{ n-2}{2n-1}F^\prime$
is anti-nef if and only if for any irreducible curve $C\subset
\mathbb{F}_n$, assuming that $C\equiv \alpha h+\beta f$, we have
\begin{align}\label{ex1}
 \mbox{mult}_{p_1}C+\mbox{mult}_{p_2}C\leqslant 2h.C.
\end{align}

If $C=S_n$ or $F$, then (\ref{ex1}) holds obviously.

If $C\neq S_n$ or $F$, then $\alpha, \beta\geqslant 0$, and
\begin{align*}
\mbox{mult}_{p_1}C+\mbox{mult}_{p_2}C \leqslant {}& F.C =  \alpha
\leqslant  2(n\alpha+\beta) =2h.C.
\end{align*}

We complete the proof.
\end{proof}

\begin{thm}\label{example2}
Let $X$ be a smooth surface which is given by blowing up $k$ points
on $\mathbb{F}_n$, $n\geqslant 2$. Assume that these points are not
on $S_n$, and no two of them are on the same fiber  of $\mathbb{F}_n \rightarrow \mathbb{P}^1$. Denote by $S_{n}^\prime$
the strict transform of $S_{n}$ on $X$. Then

(1) if $k\leqslant n+2$, then $(X,(1-\frac{2}{n})S_n^\prime)$ is a
weak log del Pezzo surface;

(2) if $n+2< k< \frac{1}{n}(n+2)^2$ and the points are in general
position, then $(X,(1-\frac{2}{n})S_n^\prime)$ is a weak log del
Pezzo surface;

(3) if $k\geqslant \frac{1}{n}(n+2)^2$, then $(X,\Delta)$ can not be
a weak log del Pezzo surface for any boundary $\Delta$.
\end{thm}

\begin{remark}
We say that $k$ points $p_1,\ldots,p_k \in \mathbb{F}_n$ are
\emph{in general position} if they satisfy the following conditions:

(1) these points are not on $S_n$, and no two of them are on the
same fiber of $\mathbb{F}_n \rightarrow \mathbb{P}^1$;

(2) the condition these points provide to define a curve $C$ is
linearly independent in $H^0(\mathcal{O}(C))$;

(3) if $n\geqslant 4$ and $k=n+4$, there does not exist a curve
$C\subset \mathbb{F}_n$ with $C\equiv \frac{\beta+1}{2}h+\beta f$,
$\beta<\frac{n}{2}$, such that $\mbox{mult}_{p_i}C=\frac{\beta+1}{2}
$ for all $1\leqslant i\leqslant k$;

(4) if $n=3$ and $k\geqslant 7$, no seven of the points lie on a
curve $C\equiv h+f$.
\end{remark}

\begin{proof}[Proof of Theorem \ref{example2}]
(3) is by Corollary \ref{cor}.

For (1)(2), the only thing to prove is the anti-nefness of
$K_X+(1-\frac{2}{n})S_n^\prime$.

Take an irreducible curve $C^\prime$ in $X$, which is not
exceptional over $\mathbb{F}_n$, and take $C$ to be the push-forward
of $C^\prime$. Let $\phi$ be the blow-up at $p_1,\ldots,p_k \in
\mathbb{F}_n$. Then
\begin{align*}
{}&\Big(K_X+\Big(1-\frac{2}{n}\Big)S_n^\prime\Big).C^\prime\\
=
{}&\Big(\phi^*\Big(K_{\mathbb{F}_n}+\Big(1-\frac{2}{n}v)S_n\Big)+\sum_{i=1}^k
E_i\Big).C^\prime\\
=
{}&\Big(K_{\mathbb{F}_n}+\Big(1-\frac{2}{n}\Big)S_n\Big).C+\sum_{i=1}^k\mbox{mult}_{p_i}C\\
={}&-\Big(1+\frac{2}{n}\Big)h.C+\sum_{i=1}^k\mbox{mult}_{p_i}C.
\end{align*}
Hence $K_X+(1-\frac{2}{n})S_n^\prime$ is anti-nef if and only if for
any irreducible curve $C\subset \mathbb{F}_n$, assuming that
$C\equiv \alpha h+\beta f$, we have
\begin{align}\label{star2}
 \sum_{i=1}^k\mbox{mult}_{p_i}C \leqslant \Big(1+\frac{2}{n}\Big)h.C.
\end{align}

There are some easy facts:
\begin{fact}
(1) If $C=S_n$, then (\ref{star2}) holds.

(2) If $C=F_{p_j}$ for some $1\leqslant j \leqslant k$, then
(\ref{star2}) holds.

(3) If $2\alpha \leqslant \beta$ and $k< \frac{1}{n}(n+2)^2$, then
(\ref{star2}) holds.

(4) If $k \leqslant n+2$,  then (\ref{star2}) holds.
\end{fact}

\begin{proof}[Proof of Fact]~

(1) Obvious.

(2) Obvious.

Hence from now on we may assume that $C\neq S_n$ and $C\neq F_{p_j}$
for $1\leqslant j \leqslant k$.

(3) We have
\begin{align*}
 {}&\sum_{i=1}^k\mbox{mult}_{p_i}C \\
\leqslant {}& \sum_{i=1}^k C.F_{p_i} =   k\alpha \leqslant
\frac{(n+2)^2}{n}\alpha
\leqslant (n+2)\alpha+\Big(1+\frac{2}{n}\Big)\beta\\
= {}& \Big(1+\frac{2}{n}\Big)h.C.
\end{align*}

(4) We have
\begin{align*}
 {}&\sum_{i=1}^k\mbox{mult}_{p_i}C \\
\leqslant {}& \sum_{i=1}^k C.F_{p_i} =   k\alpha \leqslant
(n+2)\alpha
\leqslant  (n+2)\alpha+\Big(1+\frac{2}{n}\Big)\beta\\
= {}& \Big(1+\frac{2}{n}\Big)h.C.
\end{align*}
\end{proof}

(1) of the theorem comes from Fact (4).

Note that $h^0(\mathbb{F}_n,
\mathcal{O}(C))=(\alpha+1)(\beta+1)+\frac{1}{2}(\alpha+1)\alpha n$.
Denote $\mbox{mult}_{p_i}C$ by $a_i$. To prove (2), we only need to
prove the following claim.

\begin{claim} Assume $k< \frac{(n+2)^2}{n}$, $2\alpha \geqslant \beta
+1$ and
$$
\sum_{i=1}^ka_i> (n+2)\alpha+\Big(1+\frac{2}{n}\Big)\beta.
$$
Then
$$
\sum_{i=1}^k\frac{a_i(a_i+1)}{2}>(\alpha+1)(\beta+1)+\frac{(\alpha+1)\alpha}{2}n
$$
for $p_1, p_2, \ldots, p_k$  in general position.
\end{claim}

Assuming Claim for the moment, suppose that (2) is not true, i.e.
(\ref{star2}) fails. Then
$$
\sum_{i=1}^ka_i=\sum_{i=1}^k\mbox{mult}_{p_i}C>\Big(1+\frac{2}{n}\Big)h.C=(n+2)\alpha+\Big(1+\frac{2}{n}\Big)\beta.
$$
By Claim, we have
$$
\sum_{i=1}^k\frac{a_i(a_i+1)}{2}>(\alpha+1)(\beta+1)+\frac{(\alpha+1)\alpha}{2}n.
$$
Here the left hand side is the number of conditions on $C$, and the
right hand side is $h^0(\mathbb{F}_n, \mathcal{O}(C))$. Hence we
have more conditions than sections, which is a contradiction for
$p_1, p_2, \ldots, p_k$  in general position.
\end{proof}

\begin{proof}[Proof of Claim]~

\textbf{Case 1: $\beta \geqslant n$.}

In this case, we have
\begin{align*}
{}&\sum_{i=1}^k a_i^2 + \sum_{i=1}^k a_i-2(\alpha+1)(\beta+1)-(\alpha+1)\alpha n\\
\geqslant {}&
\frac{((n+2)\alpha+(1+\frac{2}{n})\beta)^2}{k}+(n+2)\alpha+\Big(1+\frac{2}{n}\Big)\beta-2(\alpha+1)(\beta+1)-(\alpha+1)\alpha
n\\
> {}&
\frac{((n+2)\alpha+\Big(1+\frac{2}{n}\Big)\beta)^2}{\frac{(n+2)^2}{n}}+(n+2)\alpha+\Big(1+\frac{2}{n}\Big)\beta-2(\alpha+1)(\beta+1)-(\alpha+1)\alpha
n\\
= {}& \Big(\frac{\beta}{n}-1\Big)(\beta+2) \geqslant 0.
\end{align*}

\textbf{Case 2: $\frac{n}{2}\leqslant \beta < n$.}

In this case, we have $\sum_{i=1}^ka_i \geqslant (n+2)\alpha+\beta+2$ and 
\begin{align*}
 {}&\sum_{i=1}^k a_i^2 + \sum_{i=1}^k a_i-2(\alpha+1)(\beta+1)-(\alpha+1)\alpha n\\
\geqslant {}&
\frac{((n+2)\alpha+\beta+2)^2}{k}+(n+2)\alpha+\beta+2-2(\alpha+1)(\beta+1)-(\alpha+1)\alpha
n\\
> {}&
\frac{((n+2)\alpha+\beta+2)^2}{\frac{(n+2)^2}{n}}+(n+2)\alpha+\beta+2-2(\alpha+1)(\beta+1)-(\alpha+1)\alpha
n\\
= {}& \frac{n-\beta}{n+2}\Big(4\alpha-\frac{n\beta-4}{n+2}\Big)\\
= {}& \frac{(n-\beta)((4\alpha-\beta)n+8\alpha+4))}{(n+2)^2}
> 0.
\end{align*}

\textbf{Case 3: $\beta < \frac{n}{2}$. }

In this case we may assume that
$$
\sum_{i=1}^ka_i =(n+2)\alpha+\beta+1,
$$
and we need to show that
$$
\sum_{i=1}^k a_i^2>2\alpha\beta+\beta+1+\alpha^2n.
$$

\textbf{Subcase 3.1: $n\geqslant 4$.}

In this subcase we may assume that $k=n+4$. We have
\begin{align*}
 {}&\sum_{i=1}^k
a_i^2-2\alpha\beta-\beta-1-\alpha^2n\\
\geqslant {}&
\frac{((n+2)\alpha+\beta+1)^2}{n+4} -2\alpha\beta-\beta-1-\alpha^2n\\
= {}& \frac{1}{n+4}(2\alpha-\beta-1)(2\alpha-\beta+n+3) \geqslant 0.
\end{align*}
The equality holds if and only if $2\alpha=\beta+1$ and $a_i=\alpha$
for $1\leqslant i \leqslant k=n+4$. At this time, we have
$$
h^0(\mathbb{F}_n,
\mathcal{O}(C))=(\alpha+1)(\beta+1)+\frac{(\alpha+1)\alpha}{2}n,
$$
and
$$
\sum_{i=1}^k\frac{a_i(a_i+1)}{2}=(\alpha+1)(\beta+1)+\frac{(\alpha+1)\alpha}{2}n.
$$
Since $p_1, p_2, \ldots, p_k$ are  in general position, this won't
happen.

\textbf{Subcase 3.2: $n=3$.}

In this subcase we may assume that $k=8$. $\beta<\frac{n}{2}$
implies $\beta=0$ or $1$.

If $\beta=0$, then
\begin{align*}
 {}&\sum_{i=1}^k
a_i^2-2\alpha\beta-\beta-1-\alpha^2n\\
\geqslant {}& \frac{(5\alpha+1)^2}{8}-1-3\alpha^2 =
\frac{1}{8}(\alpha^2+10\alpha-7)
> 0.
\end{align*}

If $\beta=1$ and $\alpha \geqslant 2$, then
\begin{align*}
 {}&\sum_{i=1}^k
a_i^2-2\alpha\beta-\beta-1-\alpha^2n\\
\geqslant {}& \frac{(5\alpha+2)^2}{8}-2\alpha-2-3\alpha^2 =
\frac{1}{8}(\alpha^2+4\alpha-12) \geqslant 0.
\end{align*}

The equality holds if and only if $\alpha = 2$ and
$a_i=\frac{12}{8}$ for $1\leqslant i \leqslant 8$. But this will not
happen since $a_i\in \mathbb{N}$.

If $\beta=1$ and $\alpha =1$, then
$h^0(\mathbb{F}_3,\mathcal{O}(h+f))=7$. For $p_1, p_2, \ldots, p_8$
in general position, $a_i>0$ holds for at most $7$ points, hence we
may assume that $a_8=0$. Then
\begin{align*}
  {}&\sum_{i=1}^k
a_i^2-2\alpha\beta-\beta-1-\alpha^2n\\
\geqslant {}&
\frac{(5\alpha+\beta+1)^2}{7}-2\alpha\beta-\beta-1-3\alpha^2 =  0.
\end{align*}
The equality holds if and only if $a_i=1$ for $1\leqslant i
\leqslant 7$. At this time, we have
$$
h^0(\mathbb{F}_3, \mathcal{O}(C))=7,
$$
and
$$
\sum_{i=1}^8\frac{a_i(a_i+1)}{2}=7.
$$
Since $p_1, p_2, \ldots, p_k$ are in general position, this won't
happen.

\textbf{Subcase 3.3: $n=2$.}

In this subcase we may assume that $k=7$. $\beta<\frac{n}{2}$
implies $\beta=0$. We have
\begin{align*}
  {}&\sum_{i=1}^k
a_i^2-2\alpha\beta-\beta-1-\alpha^2n\\
\geqslant {}& \frac{(4\alpha +1)^2}{7} -1-2\alpha^2 =
\frac{1}{7}(2\alpha^2+8\alpha-6)
> 0.
\end{align*}
\end{proof}

We also have similar examples for $n=0,1$.
\begin{example}
Let $X$ be a smooth surface which is given by blowing up at $k$
points on $\mathbb{F}_n$, $n=0$ or $1$. Then

(1) if $k\leqslant 7$ and the points are in general position, then
$(X,0)$ is a weak log del Pezzo surface;

(2) if $k\geqslant 8$, then $(X,\Delta)$ can not be a weak log del
Pezzo surface for any boundary $\Delta$.
\end{example}

Here when $n=1$, the examples are just weak del Pezzo surfaces in
the common sense, which is well-known, and the meaning for "general
position" can be stated more precisely. Hence the examples we give
here can be viewed as some kind of generalization of the traditional
examples.

Taking $n=\lfloor 2/\epsilon \rfloor$ in Theorem \ref{example2}, it
gives examples of $\epsilon$-lc weak log del Pezzo surfaces with
$\rho({X_\textrm{min}})=O(\frac{1}{\epsilon})$. So it is interesting
to ask for the existence of $\epsilon$-lc weak log del Pezzo
surfaces with large $\rho({X_\textrm{min}})$.

\begin{q}
Is there any $\epsilon$-lc weak log del Pezzo surface $(X, \Delta)$
with $\rho({X_{\rm min}})=O(\frac{1}{\epsilon^c})$ for
$c\geqslant2$?
\end{q}

\section*{Acknowledgment} The author would like to express his deep
gratitude to his supervisor Professor Yujiro Kawamata for many
suggestions, discussions, warm encouragement and support during two
years of his master course. The author would like to thank Dr. Ching-Jui
Lai for explaining some details in \cite{Lai} patiently. The author
would like to thank Mr. Pu Cao and Dr. Yifei Chen for some useful
discussions. The author would like to thank Atsushi Ito for carefully reading this manuscript and for helpful comments.
The author is grateful to the University of Tokyo for
Special Scholarship for International Students (Todai Fellowship). The auther is supported by the Grant-in-Aid for Scientiﬁc Research (KAKENHI No. 25-6549) and the Grant-in-Aid for JSPS fellows and  Leading Graduate Course for Frontiers of Mathematical Sciences and Physics.


\begin{thebibliography}{9}


\bibitem{A} V. Alexeev, \emph{Boundedness and $K^2$ for log surfaces}, Internat. J. Math.,
5(1994), pp. 779-810.


\bibitem{AM} V. Alexeev and S. Mori, \emph{Bounding singular surfaces of general
type}, in Algebra, arithmetic and geometry with applications (West
Lafayette, IN, 2000), Springer, Berlin, 2004, pp. 143-174.

\bibitem{BCHM} C. Birkar, P. Cascini, C. D. Hacon and J. M$^c$Kernan, \emph{Existence
of minimal models for varieties of log general type}, J. Amer. Math.
Soc. 23 (2010), 405-468.

\bibitem{HM} C. Hacon and J. M$^c$Kernan, \emph{On the existence of flips}, arXiv:math/0507597.

\bibitem{H} R. Hartshorne, \emph{Algebraic geometry}, Springer-Verlag, New York, 1977.
Graduate Texts in Mathematics, No. 52.

\bibitem{KM} J. Koll\'{a}r and S. Mori, \emph{Birational geometry of algebraic varieties},
Cambridge tracts in mathematics, vol. 134, Cambridge University
Press, 1998.

\bibitem{Lai} C-J. Lai, \emph{Bounding the volumes of singular Fano
threefolds}, arXiv:1204.2593v1.

\end{thebibliography}
\end{document}